\documentclass[12pt]{article}

\usepackage{amsfonts,amsmath,amsthm,amssymb,amscd}
\usepackage[dvips]{graphicx}

\theoremstyle{definition}
\newtheorem{theorem}{Theorem}
\newtheorem{lemma}[theorem]{Lemma}
\newtheorem{proposition}[theorem]{Proposition}
\newtheorem{corollary}[theorem]{Corollary}

\numberwithin{equation}{section}
\numberwithin{theorem}{section}

\begin{document}

\begin{center}
{\bf{\Large Analogues of Jacobi's derivative formula}}
\end{center}

\begin{center}
By Kazuhide Matsuda
\end{center}

\begin{center}
Faculty of Fundamental Science, National Institute of Technology, Niihama College,\\
7-1 Yagumo-chou, Niihama, Ehime, Japan, 792-8580. \\
E-mail: matsuda@sci.niihama-nct.ac.jp  \\
Fax: 0897-37-7809 
\end{center}

{\bf Abstract}
In this paper, 
we obtain analogues of Jacobi's derivative formula in terms of the theta constants with rational characteristics. 
For this purpose, we use the arithmetic formulas of the number of representations of a natural number $n,\,\,(n=1,2,\ldots)$ as the sum of two squares, or 
the sum of a square and twice a square, 
which is given by 
\begin{align*}
S_2(n)=&\sharp \{(x,y)\in\mathbb{Z}^2  \,| \, n=x^2+y^2 \}   
        =4\sum_{d|n, \,d:odd} (-1)^{\frac{d-1}{2}},  \\
S_{1,2}(n)=&\sharp \{(x,y)\in\mathbb{Z}^2  \,| \, n=x^2+2y^2 \}
            =2(d_{1,8}(n)+d_{3,8}(n)-d_{5,8}(n)-d_{7,8}(n)),  
\end{align*}
where for the positive integers $j,k,n,$ 
$d_{j,k}(n)$ denotes the number of positive divisors $d$ of $n$ such that $d\equiv j \,\,(\mathrm{mod} \,k).$
\newline
{\bf Key Words:} theta functions; rational characteristics; Jacobi's derivative formula; the sum of two squares.
\newline
{\bf MSC(2010)}  14K25;  11E25

\section{Introduction}
\label{intro}
Throughout this paper, 
set $\mathbb{N}=\{1,2,3,\ldots\}$ and $\mathbb{N}_0=\{0,1,2,3,\ldots\}.$ 
$\mathbb{Z}$ denotes the set of rational integers. 
Furthermore, for the positive integers $j,k,n,$ 
$d_{j,k}(n)$ denotes the number of positive divisors $d$ of $n$ such that $d\equiv j \,\,(\mathrm{mod} \,k).$
\par
Following Farkas and Kra \cite{Farkas-Kra}, 
we introduce the theta function with characteristics, 
which is defined by 
\begin{equation*}
\theta 
\left[
\begin{array}{c}
\epsilon \\
\epsilon^{\prime}
\end{array}
\right] (\zeta, \tau) :=\sum_{n\in\mathbb{Z}} \exp
\left(2\pi i\left[ \frac12\left(n+\frac{\epsilon}{2}\right)^2 \tau+\left(n+\frac{\epsilon}{2}\right)\left(\zeta+\frac{\epsilon^{\prime}}{2}\right) \right] \right), 
\end{equation*}
where $\epsilon, \epsilon^{\prime}\in\mathbb{R}, \, \zeta\in\mathbb{C},$ and $\tau\in\mathbb{H}^{2},$ the upper half plane.  
The theta constants are given by 
\begin{equation*}
\theta 
\left[
\begin{array}{c}
\epsilon \\
\epsilon^{\prime}
\end{array}
\right]
:=
\theta 
\left[
\begin{array}{c}
\epsilon \\
\epsilon^{\prime}
\end{array}
\right] (0, \tau), \quad
\theta^{\prime} 
\left[
\begin{array}{c}
\epsilon \\
\epsilon^{\prime}
\end{array}
\right]
:=\left.
\frac{\partial}{\partial \zeta} 
\theta 
\left[
\begin{array}{c}
\epsilon \\
\epsilon^{\prime}
\end{array}
\right] (\zeta, \tau)
\right|_{\zeta=0}.
\end{equation*}
\par
Farkas and Kra \cite{Farkas-Kra}
treated the theta constants with rational characteristics, 
that is, 
the case where $\epsilon$ and $\epsilon^{\prime}$ are both rational numbers, 
and 
derived a number of interesting theta constant identities. 
\par
Our concern is with Jacobi's derivative formula, which is given by 
\begin{equation}
\label{eqn:Jacobi}
\theta^{\prime} 
\left[
\begin{array}{c}
1 \\
1
\end{array}
\right] 
=
-\pi 
\theta
\left[
\begin{array}{c}
0 \\
0
\end{array}
\right] 
\theta
\left[
\begin{array}{c}
1 \\
0
\end{array}
\right] 
\theta
\left[
\begin{array}{c}
0 \\
1
\end{array}
\right],  
\end{equation}
which implies that for $q\in\mathbb{C}$ with $|q|<1,$ 
$$
\prod_{n=1}^{\infty} (1-q^n)^3=\sum_{n=0}^{\infty}
(-1)^n(2n+1)q^{n(n+1)/2}. 
$$
This is equivalent to the following identity:
$$
\eta^3(\tau)=\sum_{n=0}^{\infty}
\left(
\frac{-1}{n}
\right)
n 
\exp
\left(
\frac{2\pi i n^2 \tau}{8}
\right),
$$
where 
$
\displaystyle
\eta(\tau)
=
q^{\frac{1}{24}}
\prod_{n=1}^{\infty}(1-q^n)$ with $q=\exp(2\pi i \tau), $ and 
\begin{equation*}
\left(\frac{-1}{n}\right)
=
\begin{cases}
+1, \,\,&\text{if}  \,\,n\equiv 1  \,\,(\mathrm{mod} \, 4),  \\
-1, \,\,&\text{if}  \,\,n\equiv 3 \,\,(\mathrm{mod} \, 4),  \\
0, \,\,&\text{if}  \,\,n\equiv 0 \,\,(\mathrm{mod} \, 2).  \\
\end{cases}
\end{equation*}
For generalizations of this derivative formula to a higher genus, 
see Igusa \cite{Igusa}. 
\par
Farkas \cite{Farkas-2} derived the following theta constant identity:
\begin{align}
&
\frac{
6
\theta^{\prime} 
\left[
\begin{array}{c}
1 \\
\frac13
\end{array}
\right] 
(0,\tau)
}
{
\zeta_6
\theta^3
\left[
\begin{array}{c}
\frac13 \\
\frac13
\end{array}
\right] 
(0,\tau)
+
\theta^3
\left[
\begin{array}{c}
\frac13 \\
1
\end{array}
\right] 
(0,\tau)
+\zeta_6^5
\theta^3
\left[
\begin{array}{c}
\frac13 \\
\frac53
\end{array}
\right] 
(0,\tau)
}    \notag \\
=&
\frac{2\pi i x^{\frac{1}{12}}}
{
\displaystyle
\prod_{n=0}^{\infty} (1-x^{3n+1})(1-x^{3n+2})
}   
=
2\pi i \frac{e^{\frac{\pi i}{6}}}{\sqrt{3}}
\frac
{
\theta
\left[
\begin{array}{c}
1 \\
\frac13
\end{array}
\right] 
(0,\tau)
}
{
\theta
\left[
\begin{array}{c}
\frac13 \\
1
\end{array}
\right] 
(0,3\tau)
}  
=
2\pi i
\frac
{
\theta
\left[
\begin{array}{c}
\frac13 \\
1
\end{array}
\right] 
(0,9\tau)
}
{
\theta
\left[
\begin{array}{c}
\frac13 \\
1
\end{array}
\right] 
(0,3\tau)
},  \label{eqn:Farkas} 
\end{align}
where $x=\exp(2\pi i \tau)$ and $\displaystyle\zeta_6=\exp\left(2\pi i/6\right).$ 
The identity (\ref{eqn:Farkas}) can be viewed as an analogue of Jacobi's derivative formula. 
\par
In this paper, 
we express 
$
\theta^{\prime} 
\left[
\begin{array}{c}
1 \\
1/2
\end{array}
\right], 
\theta^{\prime} 
\left[
\begin{array}{c}
1 \\
1/4
\end{array}
\right]
$ 
\normalsize{
and 
}
$
\theta^{\prime} 
\left[
\begin{array}{c}
1 \\
3/4
\end{array}
\right]
$ 
by the theta constants with rational characteristics. 
For this purpose, 
we note that 
\begin{equation}
\theta^{2}
\left[
\begin{array}{c}
0 \\
0
\end{array}
\right](0,\tau)
=
\left(    
\sum_{n\in\mathbb{Z}} x^{n^2}
\right)^2  
=
1+\sum_{n=1}^{\infty} S_2(n) x^n, \,\,x=\exp(\pi i \tau),  
\end{equation}
and
\begin{equation}
\theta
\left[
\begin{array}{c}
0 \\
0
\end{array}
\right](0,\tau)
\theta
\left[
\begin{array}{c}
0 \\
0
\end{array}
\right](0,2\tau)
=
1+\sum_{n=1}^{\infty} S_{1,2}(n) x^n, \,\,x=\exp(\pi i \tau), 
\end{equation}
where 
\begin{equation}
\label{eqn:2-squares}  
S_2(n)=\sharp \{(x,y)\in\mathbb{Z}^2  \,| \, n=x^2+y^2 \}   
        =4\sum_{d|n, \,d:odd} (-1)^{\frac{d-1}{2}},  
\end{equation}
and 
\begin{equation}
\label{eqn:1,2-squares}
S_{1,2}(n)=\sharp \{(x,y)\in\mathbb{Z}^2  \,| \, n=x^2+2y^2 \}
            =2(d_{1,8}(n)+d_{3,8}(n)-d_{5,8}(n)-d_{7,8}(n)).  
\end{equation}
For the proof of equations (\ref{eqn:2-squares}) and (\ref{eqn:1,2-squares}), see Berndt \cite[pp. 56, 74]{Berndt}. 
For the elementary proof, see Dickson \cite[pp.68]{Dickson}. 
\par
Our main theorems are as follows:
\begin{theorem}
\label{thm:general-Jacobi-1/2}
{\it
For every $\tau\in\mathbb{H}^2,$ we have 
\begin{equation}
\label{eqn:general-Jacobi-1/2-(1)}  
\theta^{\prime}
\left[
\begin{array}{c}
1 \\
\frac12
\end{array}
\right](0,\tau)
=
-\pi
\theta^{2}
\left[
\begin{array}{c}
0 \\
0
\end{array}
\right](0,2\tau)
\theta
\left[
\begin{array}{c}
1 \\
\frac12
\end{array}
\right](0,\tau), 
\end{equation}
which implies that 
\begin{equation}
\label{eqn:pro-series}
\frac
{
\eta^9(2\tau)
}
{
\eta^3(\tau)\eta^3(4\tau)
}
=
\sum_{n=0}^{\infty}
\left(\frac{-2}{n}\right)
n
\exp
\left(
\frac{2\pi i
n^2\tau
}
{8}
\right),
\end{equation}
where 
$
\displaystyle
\eta(\tau)
=
q^{\frac{1}{24}}
\prod_{n=1}^{\infty}(1-q^n)$ with $q=\exp(2\pi i \tau), $ and 
\begin{equation*}
\left(\frac{-2}{n}\right)
=
\begin{cases}
+1, \,\,&\text{if}  \,\,n\equiv 1 \,\text{or} \, \,3 \,\,(\mathrm{mod} \, 8),  \\
-1, \,\,&\text{if}  \,\,n\equiv 5 \,\text{or} \, \,7 \,\,(\mathrm{mod} \, 8),  \\
0, \,\,&\text{if}  \,\,n\equiv 0 \,\,(\mathrm{mod} \, 2).  \\
\end{cases}
\end{equation*}
}
\end{theorem}

\begin{theorem}
\label{thm:general-Jacobi-1/4,3/4}
{\it
For every $\tau\in\mathbb{H}^2,$ 
we have 
\begin{equation*}
\theta^{\prime}
\left[
\begin{array}{c}
1 \\
\frac14
\end{array}
\right](0,\tau)
=
-
\pi
\theta
\left[
\begin{array}{c}
1 \\
\frac14
\end{array}
\right](0,\tau)
\theta
\left[
\begin{array}{c}
0 \\
0
\end{array}
\right](0,4\tau)
\left\{
\sqrt{2}
\theta
\left[
\begin{array}{c}
0 \\
0
\end{array}
\right](0,2\tau)
-
\theta
\left[
\begin{array}{c}
0 \\
0
\end{array}
\right](0,4\tau)
\right\},  
\end{equation*}
and
\begin{equation*}
\theta^{\prime}
\left[
\begin{array}{c}
1 \\
\frac34
\end{array}
\right](0,\tau)
=
-
\pi
\theta
\left[
\begin{array}{c}
1 \\
\frac34
\end{array}
\right](0,\tau)
\theta
\left[
\begin{array}{c}
0 \\
0
\end{array}
\right](0,4\tau)
\left\{
\sqrt{2}
\theta
\left[
\begin{array}{c}
0 \\
0
\end{array}
\right](0,2\tau)
+
\theta
\left[
\begin{array}{c}
0 \\
0
\end{array}
\right](0,4\tau)
\right\}.
\end{equation*}
}
\end{theorem}
Equation (\ref{eqn:pro-series}) was proved by Zucker \cite{Zucker1}, \cite{Zucker2} and 
K\"ohler \cite{Kohler}. 
Noted is that they did not use theta functions with rational characteristics. 
Therefore, 
we believe that equation (\ref{eqn:general-Jacobi-1/2-(1)}) is new. 
For the theory of the eta products and theta series identities, 
see K\"ohler \cite{Kohler2}. 
\par
This paper is organized as follows. In Section 2, 
we review the properties of the theta functions. 
In Sections 3 and 4, we prove Theorems \ref{thm:general-Jacobi-1/2} and \ref{thm:general-Jacobi-1/4,3/4}. 
In Section 3, we especially note that equation (\ref{eqn:general-Jacobi-1/2-(1)}) has a relationship with 
a certain kind of partion number. 
In Section 5, we derive more theta constant identities.


\subsubsection*{Acknowledgments}
We are grateful to Professor Nakayashiki and Professor Nishiyama for their useful comments. 
We also thank the referee for recommending various improvements to the paper.

\section{The properties of the theta functions}
\label{sec:properties}

\subsection{Basic properties}
We first note that 
for $m,n\in\mathbb{Z},$ 
\begin{equation}
\label{eqn:integer-char}
\theta 
\left[
\begin{array}{c}
\epsilon \\
\epsilon^{\prime}
\end{array}
\right] (\zeta+n+m\tau, \tau) =
\exp(2\pi i)\left[\frac{n\epsilon-m\epsilon^{\prime}}{2}-mz-\frac{m^2\tau}{2}\right]
\theta 
\left[
\begin{array}{c}
\epsilon \\
\epsilon^{\prime}
\end{array}
\right] (\zeta,\tau),
\end{equation}
and 
\begin{equation}
\theta 
\left[
\begin{array}{c}
\epsilon +2m\\
\epsilon^{\prime}+2n
\end{array}
\right] 
(\zeta,\tau)
=\exp(\pi i \epsilon n)
\theta 
\left[
\begin{array}{c}
\epsilon \\
\epsilon^{\prime}
\end{array}
\right] 
(\zeta,\tau).
\end{equation}
Furthermore, 
\begin{equation*}
\theta 
\left[
\begin{array}{c}
-\epsilon \\
-\epsilon^{\prime}
\end{array}
\right] (\zeta,\tau)
=
\theta 
\left[
\begin{array}{c}
\epsilon \\
\epsilon^{\prime}
\end{array}
\right] (-\zeta,\tau)
\,\,
\mathrm{and}
\,\,
\theta^{\prime} 
\left[
\begin{array}{c}
-\epsilon \\
-\epsilon^{\prime}
\end{array}
\right] (\zeta,\tau)
=
-
\theta^{\prime} 
\left[
\begin{array}{c}
\epsilon \\
\epsilon^{\prime}
\end{array}
\right] (-\zeta,\tau).
\end{equation*}
\par
For $m,n\in\mathbb{R},$ 
we see that 
\begin{align}
\label{eqn:real-char}
&\theta 
\left[
\begin{array}{c}
\epsilon \\
\epsilon^{\prime}
\end{array}
\right] \left(\zeta+\frac{n+m\tau}{2}, \tau\right)   \notag\\
&=
\exp(2\pi i)\left[
-\frac{mz}{2}-\frac{m^2\tau}{8}-\frac{m(\epsilon^{\prime}+n)}{4}
\right]
\theta 
\left[
\begin{array}{c}
\epsilon+m \\
\epsilon^{\prime}+n
\end{array}
\right] 
(\zeta,\tau). 
\end{align}
We note that 
$\theta 
\left[
\begin{array}{c}
\epsilon \\
\epsilon^{\prime}
\end{array}
\right] \left(\zeta, \tau\right)$ has only one zero in the fundamental parallelogram, 
which is given by 
$
\displaystyle
\zeta=\frac{1-\epsilon}{2}\tau+\frac{1-\epsilon^{\prime}}{2}. 
$

\subsection{Jacobi's triple product identity}

All the theta functions have infinite product expansions, which are given by 
\begin{align}
\theta 
\left[
\begin{array}{c}
\epsilon \\
\epsilon^{\prime}
\end{array}
\right] (\zeta, \tau) &=\exp\left(\frac{\pi i \epsilon \epsilon^{\prime}}{2}\right) x^{\frac{\epsilon^2}{4}} z^{\frac{\epsilon}{2}}    \notag  \\
                           &\quad 
                           \displaystyle \prod_{n=1}^{\infty}(1-x^{2n})(1+e^{\pi i \epsilon^{\prime}} x^{2n-1+\epsilon} z)(1+e^{-\pi i \epsilon^{\prime}} x^{2n-1-\epsilon}/z),  \label{eqn:Jacobi-triple}
\end{align}
where $x=\exp(\pi i \tau)$ and $z=\exp(2\pi i \zeta).$

\subsection{Spaces of $N$-th order $\theta$-functions}

Following Farkas and Kra \cite{Farkas-Kra}, 
we define 
$\mathcal{F}_{N}\left[
\begin{array}{c}
\epsilon \\
\epsilon^{\prime}
\end{array}
\right] $ to be the set of entire functions $f$ that satisfy the two functional equations, 
$$
f(\zeta+1)=\exp(\pi i \epsilon) \,\,f(\zeta),
$$
and 
$$
f(\zeta+\tau)=\exp(-\pi i)[\epsilon^{\prime}+2N\zeta+N\tau] \,\,f(\zeta), \quad \zeta\in\mathbb{C},  \,\,\tau \in\mathbb{H}^2,
$$ 
where 
$N$ is a positive integer and 
$\left[
\begin{array}{c}
\epsilon \\
\epsilon^{\prime}
\end{array}
\right] \in\mathbb{R}^2.$ 
This set of functions is called the space of {\it $N$-th order $\theta$-functions with characteristic }
$\left[
\begin{array}{c}
\epsilon \\
\epsilon^{\prime}
\end{array}
\right]. $ 
Note that 
$$
\dim \mathcal{F}_{N}\left[
\begin{array}{c}
\epsilon \\
\epsilon^{\prime}
\end{array}
\right] =N.
$$
For its proof, see Farkas and Kra \cite[pp.133]{Farkas-Kra}.

\subsection{Lemma of Farkas and Kra}

We recall the lemma of Farkas and Kra \cite[pp.78]{Farkas-Kra}. 

\begin{lemma}
\label{lem:Farkas-Kra}
{\it
For 
all characteristics 
$
\left[
\begin{array}{c}
\epsilon \\
\epsilon^{\prime}
\end{array}
\right], 
\left[
\begin{array}{c}
\delta \\
\delta^{\prime}
\end{array}
\right]
$ 
and 
all $\tau\in\mathbb{H}^2,$ 
we have
\begin{align*}
&\theta
\left[
\begin{array}{c}
\epsilon \\
\epsilon^{\prime}
\end{array}
\right](0,\tau)
\theta
\left[
\begin{array}{c}
\delta \\
\delta^{\prime}
\end{array}
\right](0,\tau)  \\
=&
\theta
\left[
\begin{array}{c}
\frac{\epsilon+\delta}{2} \\
\epsilon^{\prime}+\delta^{\prime}
\end{array}
\right](0,2\tau)
\theta
\left[
\begin{array}{c}
\frac{\epsilon-\delta}{2} \\
\epsilon^{\prime}-\delta^{\prime}
\end{array}
\right](0,2\tau)  
+
\theta
\left[
\begin{array}{c}
\frac{\epsilon+\delta}{2}+1 \\
\epsilon^{\prime}+\delta^{\prime}
\end{array}
\right](0,2\tau)
\theta
\left[
\begin{array}{c}
\frac{\epsilon-\delta}{2}+1 \\
\epsilon^{\prime}-\delta^{\prime}
\end{array}
\right](0,2\tau). 
\end{align*}
}
\end{lemma}

\section{On Theorem \ref{thm:general-Jacobi-1/2}}

\subsection{Proof of equation (\ref{eqn:general-Jacobi-1/2-(1)})}

Using equation (\ref{eqn:2-squares}), 
we prove Theorem \ref{thm:general-Jacobi-1/2}.

\begin{proof}
The Jacobi triple product identity (\ref{eqn:Jacobi-triple}) yields 
\begin{align*}
\frac{
\theta^{\prime}
\left[
\begin{array}{c}
1 \\
\frac12
\end{array}
\right](0,\tau)  
}
{
\theta
\left[
\begin{array}{c}
1 \\
\frac12
\end{array}
\right](0,\tau)  
}
=&
\left.
\frac{d}{d\zeta} 
\log 
\theta
\left[
\begin{array}{c}
1 \\
\frac12
\end{array}
\right](\zeta,\tau)  
\right|_{\zeta=0}  \\
=&
\left.
\frac{dz}{d\zeta} \cdot
\frac{d}{dz} 
\log 
\left(
e^{\frac{\pi i}{4}} x^{1/4} z^{1/2} 
\prod_{n=1}^{\infty}
(1-x^{2n})(1+ix^{2n} z)(1-i \frac{x^{2n-2}}{z})
\right)
\right|_{\zeta=0} \\
=&
\left.
2\pi i z
\left\{
\frac{1}{2z}
+
\sum_{n=1}^{\infty} \frac{ix^{2n}}{1+i x^{2n}z}
+
\sum_{n=1}^{\infty} \frac{-ix^{2n-2} (-z^{-2}) }{1-i x^{2n-2}z^{-1}}
\right\}
\right|_{\zeta=0} \\
=&
2\pi i
\left\{
\frac12+\frac{i}{1-i}
+\sum_{n=1}^{\infty}
 \frac{ix^{2n}}{1+ix^{2n}}
+
\sum_{n=1}^{\infty} \frac{ix^{2n}}{1-ix^{2n}}
\right\}  \\
=&
2\pi i
\left\{
\frac{i}{2}+2i \sum_{n=1}^{\infty} \frac{x^{2n}}{1+x^{4n}}
\right\}  \\
=&
-\pi 
\left\{
1+4\sum_{n=1}^{\infty} x^{2n} \sum_{m=0}^{\infty} (-1)^m x^{4nm} 
\right\}  \\
=&
-\pi 
\left\{
1+4\sum_{n=1}^{\infty} \sum_{m=0}^{\infty}  (-1)^m x^{2n(2m+1)} 
\right\}  \\
=&
-\pi 
\left\{
1+4\sum_{N=1}^{\infty} x^{2N} \left(\sum_{d | N, d:odd} (-1)^{\frac{d-1}{2}}   \right) 
\right\}  \\
=&
-\pi \sum_{n=0}^{\infty} S_2(n) x^{2n}  \\
=&
-\pi 
\theta^2
\left[
\begin{array}{c}
0 \\
0
\end{array}
\right](0,2\tau),
\end{align*}
where $x=\exp(\pi i \tau)$ and $z=\exp(2\pi i \zeta).$ 
\end{proof}

\subsection{Proof of equation (\ref{eqn:pro-series})}

\begin{proof}
Theorem \ref{thm:general-Jacobi-1/2} and Jacobi's triple product identity (\ref{eqn:Jacobi-triple}) imply that 
\begin{equation}
\label{eqn:pro-series-(1)}
\prod_{n=1}^{\infty}
\frac{(1-q^{2n})^9}{(1-q^n)^3(1-q^{4n})^3}
=
\sum_{n=0}^{\infty}
\left(\frac{-2}{n}\right)
n
q^
{
t_{(n-1)/2}
},
\end{equation}
where for $n\in\mathbb{Z},$ $t_n=n(n+1)/2.$
\par 
Replacing $q$ by $q^8$ and multiplying both sides of equation (\ref{eqn:pro-series-(1)}) by $q,$ 
we obtain 
\begin{equation*}
q
\prod_{n=1}^{\infty}
\frac{(1-q^{16n})^9}{(1-q^{8n})^3(1-q^{32n})^3}
=
\sum_{n=0}^{\infty}
\left(\frac{-2}{n}\right)
n
q^
{
n^2
}.
\end{equation*}
Setting $q=\exp(2\pi i \tau),$ 
we obtain 
\begin{equation*}
\frac
{\eta^9(16\tau)}
{\eta^3(8\tau)\eta^3(32\tau)}
=
\sum_{n=0}^{\infty}
\left(\frac{-2}{n}\right)
n
\exp(
2\pi i 
n^2
\tau
). 
\end{equation*}
Replacing $\tau$ by $\tau/8,$ we can prove equation (\ref{eqn:pro-series}). 
\end{proof}

\section{Proof of Theorem \ref{thm:general-Jacobi-1/4,3/4} }

\subsection{Applications of $S_2(n)$ and $S_{1,2}(n)$}

\begin{proposition}
\label{prop:1/4-pm-3/4}
{\it
For every $\tau\in\mathbb{H}^2,$ we have
\begin{equation}
\label{eqn-log-diff-difference-(1)}
\frac
{
\theta^{\prime}
\left[
\begin{array}{c}
1 \\
\frac14
\end{array}
\right](0,\tau)
}
{
\theta
\left[
\begin{array}{c}
1 \\
\frac14
\end{array}
\right](0,\tau)
}
-
\frac
{
\theta^{\prime}
\left[
\begin{array}{c}
1 \\
\frac34
\end{array}
\right](0,\tau)
}
{
\theta
\left[
\begin{array}{c}
1 \\
\frac34
\end{array}
\right](0,\tau)
}
=
2\pi
\theta^2
\left[
\begin{array}{c}
0 \\
0
\end{array}
\right](0,4\tau),
\end{equation}
and
\begin{equation}
\label{eqn-log-diff-difference-(2)}
\frac
{
\theta^{\prime}
\left[
\begin{array}{c}
1 \\
\frac14
\end{array}
\right](0,\tau)
}
{
\theta
\left[
\begin{array}{c}
1 \\
\frac14
\end{array}
\right](0,\tau)
}
+
\frac
{
\theta^{\prime}
\left[
\begin{array}{c}
1 \\
\frac34
\end{array}
\right](0,\tau)
}
{
\theta
\left[
\begin{array}{c}
1 \\
\frac34
\end{array}
\right](0,\tau)
}
=
-2\sqrt{2}\pi
\theta
\left[
\begin{array}{c}
0 \\
0
\end{array}
\right](0,2\tau)
\theta
\left[
\begin{array}{c}
0 \\
0
\end{array}
\right](0,4\tau).
\end{equation}
}
\end{proposition}

\begin{proof}
Jacobi's triple product identity (\ref{eqn:Jacobi-triple}) yields 
\begin{align*}
\frac{
\theta^{\prime}
\left[
\begin{array}{c}
1 \\
\frac14
\end{array}
\right](0,\tau)  
}
{
\theta
\left[
\begin{array}{c}
1 \\
\frac14
\end{array}
\right](0,\tau)  
}
=&
2\pi i 
\left\{
\frac{i}{\sqrt{2}}-\frac{i}{2}+
2i
\sum_{N=1}^{\infty} x^{2N}\left(\sum_{d|N} (-1)^{d-1} \sin \frac{\pi d}{4}    \right)
\right\},  \\
\frac{
\theta^{\prime}
\left[
\begin{array}{c}
1 \\
\frac34
\end{array}
\right](0,\tau)  
}
{
\theta
\left[
\begin{array}{c}
1 \\
\frac34
\end{array}
\right](0,\tau)  
}
=&
2\pi i 
\left\{
\frac{i}{\sqrt{2}}+\frac{i}{2}+
2i
\sum_{N=1}^{\infty} x^{2N}\left(\sum_{d|N} (-1)^{d-1} \sin \frac{3\pi d}{4}    \right)
\right\},  
\end{align*}
where $x=\exp(\pi i \tau).$ 
\par
We first treat equation (\ref{eqn-log-diff-difference-(1)}). 
For this purpose, we have  
\begin{align*}
\frac{
\theta^{\prime}
\left[
\begin{array}{c}
1 \\
\frac14
\end{array}
\right](0,\tau)  
}
{
\theta
\left[
\begin{array}{c}
1 \\
\frac14
\end{array}
\right](0,\tau)  
}
-
\frac{
\theta^{\prime}
\left[
\begin{array}{c}
1 \\
\frac34
\end{array}
\right](0,\tau)  
}
{
\theta
\left[
\begin{array}{c}
1 \\
\frac34
\end{array}
\right](0,\tau)  
}
=&
2\pi i
\left\{
-i+2i 
\sum_{N=1}^{\infty}x^{2N}
\left(
\sum_{d|N} (-1)^{d-1}
\left[
\sin \frac{\pi d}{4}-\sin \frac{3\pi d}{4}
\right]
\right)
\right\}  \\
=&
2\pi i
\left\{
-i-4i 
\sum_{N=1}^{\infty} x^{2N}
\left(
\sum_{d|N} (-1)^{d-1} \cos\frac{\pi d}{2} \sin\frac{\pi d}{4}
\right)
\right\}  \\
=&
2\pi i
\left\{
-i-4i 
\sum_{N=1}^{\infty} x^{2N}
\left(
\sum_{d|N, d\equiv 2 \,\mathrm{mod}\, 8} 1
-
\sum_{d|N, d\equiv 6 \,\mathrm{mod}\, 8} 1
\right)
\right\}  \\
=&
2\pi 
\left\{
1+4  \sum_{N=1}^{\infty} x^{4N} \left(  \sum_{d|N, \, d:odd} (-1)^{\frac{d-1}{2}} \right)
\right\}  \\
=&
2\pi 
\theta^2
\left[
\begin{array}{c}
0 \\
0
\end{array}
\right](0,4\tau). 
\end{align*}
\par
We next deal with (\ref{eqn-log-diff-difference-(2)}). 
For this purpose, we have  
\begin{align*}
\frac{
\theta^{\prime}
\left[
\begin{array}{c}
1 \\
\frac14
\end{array}
\right](0,\tau)  
}
{
\theta
\left[
\begin{array}{c}
1 \\
\frac14
\end{array}
\right](0,\tau)  
}
+
\frac{
\theta^{\prime}
\left[
\begin{array}{c}
1 \\
\frac34
\end{array}
\right](0,\tau)  
}
{
\theta
\left[
\begin{array}{c}
1 \\
\frac34
\end{array}
\right](0,\tau)  
}
=&
2\pi i
\left\{
\sqrt{2}i
+
2i 
\sum_{N=1}^{\infty}x^{2N}
\left(
\sum_{d|N} (-1)^{d-1}
\left[
\sin \frac{\pi d}{4}+\sin \frac{3\pi d}{4}
\right]
\right)
\right\}  \\
=&
2\pi i
\left\{
\sqrt{2}i
+2i 
\sum_{N=1}^{\infty} x^{2N}
\left(
\sum_{d|N} (-1)^{d-1} 2\sin\frac{\pi d}{2} \cos\frac{\pi d}{4}
\right)
\right\}  \\
=&
2\pi i
\left\{
\sqrt{2}i
+2i 
\sum_{N=1}^{\infty} x^{2N}
\sqrt{2}
\left(
d_{1,8}(N)+d_{3,8}(N)-d_{5,8}(N)-d_{7,8}(N)
\right)
\right\}  \\
=&-
2\sqrt{2}\pi 
\left\{
1+2 \sum_{N=1}^{\infty} S_{1,2}(N)x^{2N} 
\right\}  \\
=&
-
2\sqrt{2}\pi 
\theta
\left[
\begin{array}{c}
0 \\
0
\end{array}
\right](0,2\tau)
\theta
\left[
\begin{array}{c}
0 \\
0
\end{array}
\right](0,4\tau). 
\end{align*}
\end{proof}

\subsection{Proof of Theorem \ref{thm:general-Jacobi-1/4,3/4} }

\begin{proof}
The theorem follows from Proposition \ref{prop:1/4-pm-3/4}. 
\end{proof}

\subsection{Another expressions of the derivative formulas}

\begin{proposition}
\label{prop:theta-functional-formula}
{\it
For every $(\zeta,\tau)\in\mathbb{C}\times\mathbb{H}^2,$ 
we have
\begin{align}
&
\frac
{
\theta^4
\left[
\begin{array}{c}
\frac14 \\
\frac14
\end{array}
\right](\zeta,\tau) 
-
\zeta_8^3
\theta^4
\left[
\begin{array}{c}
\frac14 \\
\frac34
\end{array}
\right](\zeta,\tau) 
+
\zeta_8^6
\theta^4
\left[
\begin{array}{c}
\frac14 \\
\frac54
\end{array}
\right](\zeta,\tau) 
-
\zeta_8
\theta^4
\left[
\begin{array}{c}
\frac14 \\
\frac74
\end{array}
\right](\zeta,\tau) 
}
{
\theta
\left[
\begin{array}{c}
\frac14 \\
0
\end{array}
\right](\zeta,\tau)
 \theta
\left[
\begin{array}{c}
\frac14 \\
\frac12
\end{array}
\right](\zeta,\tau) 
\theta
\left[
\begin{array}{c}
\frac14 \\
1
\end{array}
\right](\zeta,\tau) 
\theta
\left[
\begin{array}{c}
\frac14 \\
\frac32
\end{array}
\right](\zeta,\tau) 
}    \notag  \\
=&
-
\frac
{
8\zeta_8^3
\left(
\theta^3
\left[
\begin{array}{c}
1 \\
\frac14
\end{array}
\right](0,\tau) 
\theta^{\prime}
\left[
\begin{array}{c}
1 \\
\frac14
\end{array}
\right](0,\tau)
-
\theta^3
\left[
\begin{array}{c}
1 \\
\frac34
\end{array}
\right](0,\tau) 
\theta^{\prime}
\left[
\begin{array}{c}
1 \\
\frac34
\end{array}
\right](0,\tau)
\right)
}
{
\theta^{\prime}
\left[
\begin{array}{c}
1 \\
1
\end{array}
\right](0,\tau)
\theta
\left[
\begin{array}{c}
1 \\
0
\end{array}
\right](0,\tau)
\theta^2
\left[
\begin{array}{c}
1 \\
\frac12
\end{array}
\right](0,\tau)
},  \label{eqn-1/4-half}
\end{align}
and
\begin{align}
&
\frac
{
\theta^4
\left[
\begin{array}{c}
\frac34 \\
\frac14
\end{array}
\right](\zeta,\tau) 
-
\zeta_8
\theta^4
\left[
\begin{array}{c}
\frac34 \\
\frac34
\end{array}
\right](\zeta,\tau) 
+
\zeta_8^2
\theta^4
\left[
\begin{array}{c}
\frac34 \\
\frac54
\end{array}
\right](\zeta,\tau) 
-
\zeta_8^3
\theta^4
\left[
\begin{array}{c}
\frac34 \\
\frac74
\end{array}
\right](\zeta,\tau) 
}
{
\theta
\left[
\begin{array}{c}
\frac34 \\
0
\end{array}
\right](\zeta,\tau)
 \theta
\left[
\begin{array}{c}
\frac34 \\
\frac12
\end{array}
\right](\zeta,\tau) 
\theta
\left[
\begin{array}{c}
\frac34 \\
1
\end{array}
\right](\zeta,\tau) 
\theta
\left[
\begin{array}{c}
\frac34 \\
\frac32
\end{array}
\right](\zeta,\tau) 
}    \notag  \\
=&
-
\frac
{
8\zeta_8
\left(
\theta^3
\left[
\begin{array}{c}
1 \\
\frac14
\end{array}
\right](0,\tau) 
\theta^{\prime}
\left[
\begin{array}{c}
1 \\
\frac14
\end{array}
\right](0,\tau)
-
\theta^3
\left[
\begin{array}{c}
1 \\
\frac34
\end{array}
\right](0,\tau) 
\theta^{\prime}
\left[
\begin{array}{c}
1 \\
\frac34
\end{array}
\right](0,\tau)
\right)
}
{
\theta^{\prime}
\left[
\begin{array}{c}
1 \\
1
\end{array}
\right](0,\tau)
\theta
\left[
\begin{array}{c}
1 \\
0
\end{array}
\right](0,\tau)
\theta^2
\left[
\begin{array}{c}
1 \\
\frac12
\end{array}
\right](0,\tau)
},  \label{eqn-3/4-half}
\end{align}
where $\zeta_8=\exp(2\pi i/8).$ 
}
\end{proposition}

\begin{proof}
We treat equation (\ref{eqn-1/4-half}). 
Equation (\ref{eqn-3/4-half}) can be proved in the same way. 
For this purpose, 
we set 
$$
f(\zeta)=
\frac
{
\theta^4
\left[
\begin{array}{c}
\frac14 \\
\frac14
\end{array}
\right](\zeta,\tau) 
-
\zeta_8^3
\theta^4
\left[
\begin{array}{c}
\frac14 \\
\frac34
\end{array}
\right](\zeta,\tau) 
+
\zeta_8^6
\theta^4
\left[
\begin{array}{c}
\frac14 \\
\frac54
\end{array}
\right](\zeta,\tau) 
-
\zeta_8
\theta^4
\left[
\begin{array}{c}
\frac14 \\
\frac74
\end{array}
\right](\zeta,\tau) 
}
{
\theta
\left[
\begin{array}{c}
\frac14 \\
0
\end{array}
\right](\zeta,\tau)
 \theta
\left[
\begin{array}{c}
\frac14 \\
\frac12
\end{array}
\right](\zeta,\tau) 
\theta
\left[
\begin{array}{c}
\frac14 \\
1
\end{array}
\right](\zeta,\tau) 
\theta
\left[
\begin{array}{c}
\frac14 \\
\frac32
\end{array}
\right](\zeta,\tau) 
}.   
$$
It can be easily verified that 
$f(\zeta)$ is an elliptic function with no poles, 
because both the numerator and denominator of $f(\zeta)$ have the same zeros. 
Thus, $f(\zeta)$ is constant. 
\par
If we differentiate both the numerator and denominator of $f(\zeta)$ and set 
$\zeta=\frac{3\tau}{8},$ 
we obtain equation (\ref{eqn-1/4-half}).
\end{proof}

Setting $\zeta=0$ in Proposition \ref{prop:theta-functional-formula} 
and 
using Jacobi's derivative's formula (\ref{eqn:Jacobi}),  
we obtain the following corollary:

\begin{corollary}
\label{coro:1/4-3/4}
{\it
For every $\tau\in\mathbb{H}^2,$ 
we have
\begin{align*}
&\theta^3
\left[
\begin{array}{c}
1 \\
\frac14
\end{array}
\right]
\theta^{\prime}
\left[
\begin{array}{c}
1 \\
\frac14
\end{array}
\right]
-
\theta^3
\left[
\begin{array}{c}
1 \\
\frac34
\end{array}
\right] 
\theta^{\prime}
\left[
\begin{array}{c}
1 \\
\frac34
\end{array}
\right]  \\
=&
\frac{\pi}{8\zeta_8^3}
\theta
\left[
\begin{array}{c}
0 \\
0
\end{array}
\right] 
\theta
\left[
\begin{array}{c}
0 \\
1
\end{array}
\right] 
\theta^2
\left[
\begin{array}{c}
1\\
0
\end{array}
\right] 
\theta^2
\left[
\begin{array}{c}
1\\
\frac12
\end{array}
\right] 
\frac
{
\theta^4
\left[
\begin{array}{c}
\frac14 \\
\frac14
\end{array}
\right]
-
\zeta_8^3
\theta^4
\left[
\begin{array}{c}
\frac14 \\
\frac34
\end{array}
\right]
+
\zeta_8^6
\theta^4
\left[
\begin{array}{c}
\frac14 \\
\frac54
\end{array}
\right]
-
\zeta_8
\theta^4
\left[
\begin{array}{c}
\frac14 \\
\frac74
\end{array}
\right]
}
{
\theta
\left[
\begin{array}{c}
\frac14 \\
0
\end{array}
\right]
 \theta
\left[
\begin{array}{c}
\frac14 \\
\frac12
\end{array}
\right]
\theta
\left[
\begin{array}{c}
\frac14 \\
1
\end{array}
\right]
\theta
\left[
\begin{array}{c}
\frac14 \\
\frac32
\end{array}
\right]
}   \\
=&
\frac{\pi}{8\zeta_8}
\theta
\left[
\begin{array}{c}
0 \\
0
\end{array}
\right] 
\theta
\left[
\begin{array}{c}
0 \\
1
\end{array}
\right] 
\theta^2
\left[
\begin{array}{c}
1\\
0
\end{array}
\right] 
\theta^2
\left[
\begin{array}{c}
1\\
\frac12
\end{array}
\right] 
\frac
{
\theta^4
\left[
\begin{array}{c}
\frac34 \\
\frac14
\end{array}
\right]
-
\zeta_8
\theta^4
\left[
\begin{array}{c}
\frac34 \\
\frac34
\end{array}
\right]
+
\zeta_8^2
\theta^4
\left[
\begin{array}{c}
\frac34 \\
\frac54
\end{array}
\right]
-
\zeta_8^3
\theta^4
\left[
\begin{array}{c}
\frac34 \\
\frac74
\end{array}
\right]
}
{
\theta
\left[
\begin{array}{c}
\frac34 \\
0
\end{array}
\right]
 \theta
\left[
\begin{array}{c}
\frac34 \\
\frac12
\end{array}
\right]
\theta
\left[
\begin{array}{c}
\frac34 \\
1
\end{array}
\right]
\theta
\left[
\begin{array}{c}
\frac34 \\
\frac32
\end{array}
\right]
}, 
\end{align*}
where $\zeta_8=\exp(2\pi i/8).$ 
}
\end{corollary}

By Lemma \ref{lem:Farkas-Kra} and Corollary \ref{coro:1/4-3/4}, 
we obtain the following proposition:

\begin{proposition}
\label{prop:Jacobi-1/4-(1)}
{\it
For every $\tau\in\mathbb{H}^2,$ 
we have 
\begin{align*}
&\theta^4
\left[
\begin{array}{c}
1 \\
\frac14
\end{array}
\right](0,\tau)
\cdot
\frac{
\theta^{\prime}
\left[
\begin{array}{c}
1 \\
\frac14
\end{array}
\right](0,\tau)
}
{
\theta
\left[
\begin{array}{c}
1 \\
\frac14
\end{array}
\right](0,\tau)
}
-
\theta^4
\left[
\begin{array}{c}
1 \\
\frac34
\end{array}
\right](0,\tau)
\cdot
\frac{
\theta^{\prime}
\left[
\begin{array}{c}
1 \\
\frac34
\end{array}
\right](0,\tau)
}
{
\theta
\left[
\begin{array}{c}
1 \\
\frac34
\end{array}
\right](0,\tau)
}  \\
=&
-2\pi
\theta^2
\left[
\begin{array}{c}
0 \\
0
\end{array}
\right](0,4\tau)
\theta
\left[
\begin{array}{c}
1 \\
0
\end{array}
\right](0,4\tau)
\theta
\left[
\begin{array}{c}
0 \\
1
\end{array}
\right](0,4\tau)  
\left\{
\theta^2
\left[
\begin{array}{c}
0 \\
0
\end{array}
\right](0,4\tau)
+
3
\theta^2
\left[
\begin{array}{c}
1 \\
0
\end{array}
\right](0,4\tau)
\right\}. 
\end{align*}
}
\end{proposition}

\begin{proof}
By Lemma \ref{lem:Farkas-Kra}, 
we have 
$$
\theta^2
\left[
\begin{array}{c}
\frac14 \\
\frac14
\end{array}
\right](0,\tau)
=
\theta
\left[
\begin{array}{c}
\frac14 \\
\frac12
\end{array}
\right](0,2\tau)
\theta
\left[
\begin{array}{c}
0 \\
0
\end{array}
\right](0,2\tau)
+\zeta_8^5
\theta
\left[
\begin{array}{c}
\frac34 \\
\frac32
\end{array}
\right](0,2\tau)
\theta
\left[
\begin{array}{c}
1\\
0
\end{array}
\right](0,2\tau),
$$
$$
\theta^2
\left[
\begin{array}{c}
\frac14 \\
\frac34
\end{array}
\right](0,\tau)
=
\theta
\left[
\begin{array}{c}
\frac14 \\
\frac32
\end{array}
\right](0,2\tau)
\theta
\left[
\begin{array}{c}
0 \\
0
\end{array}
\right](0,2\tau)
+\zeta_8^5
\theta
\left[
\begin{array}{c}
\frac34 \\
\frac12
\end{array}
\right](0,2\tau)
\theta
\left[
\begin{array}{c}
1\\
0
\end{array}
\right](0,2\tau),
$$
$$
\theta^2
\left[
\begin{array}{c}
\frac14 \\
\frac54
\end{array}
\right](0,\tau)
=
\zeta_8
\theta
\left[
\begin{array}{c}
\frac14 \\
\frac12
\end{array}
\right](0,2\tau)
\theta
\left[
\begin{array}{c}
0 \\
0
\end{array}
\right](0,2\tau)
+\zeta_8^2
\theta
\left[
\begin{array}{c}
\frac34 \\
\frac32
\end{array}
\right](0,2\tau)
\theta
\left[
\begin{array}{c}
1\\
0
\end{array}
\right](0,2\tau),
$$
$$
\theta^2
\left[
\begin{array}{c}
\frac14 \\
\frac74
\end{array}
\right](0,\tau)
=
\zeta_8
\theta
\left[
\begin{array}{c}
\frac14 \\
\frac32
\end{array}
\right](0,2\tau)
\theta
\left[
\begin{array}{c}
0 \\
0
\end{array}
\right](0,2\tau)
+\zeta_8^2
\theta
\left[
\begin{array}{c}
\frac34 \\
\frac12
\end{array}
\right](0,2\tau)
\theta
\left[
\begin{array}{c}
1\\
0
\end{array}
\right](0,2\tau).
$$
\par
Furthermore, by Lemma \ref{lem:Farkas-Kra}, 
we obtain
\begin{align*}
&
\theta^4
\left[
\begin{array}{c}
\frac14 \\
\frac14
\end{array}
\right](0,\tau)
-
\zeta_8^3
\theta^4
\left[
\begin{array}{c}
\frac14 \\
\frac34
\end{array}
\right](0,\tau)
+
\zeta_8^6
\theta^4
\left[
\begin{array}{c}
\frac14 \\
\frac54
\end{array}
\right](0,\tau)
-
\zeta_8
\theta^4
\left[
\begin{array}{c}
\frac14 \\
\frac74
\end{array}
\right](0,\tau)  \\
=&
2
\theta^2
\left[
\begin{array}{c}
0 \\
0
\end{array}
\right](0,2\tau)
\left\{
\theta^2
\left[
\begin{array}{c}
\frac14 \\
\frac12
\end{array}
\right](0,2\tau)
-
\zeta_8^3
\theta^2
\left[
\begin{array}{c}
\frac14 \\
\frac32
\end{array}
\right](0,2\tau)
\right\}   \\
&\hspace{15mm} 
-
2
\zeta_8^2
\theta^2
\left[
\begin{array}{c}
1 \\
0
\end{array}
\right](0,2\tau)
\left\{
\zeta_8^3
\theta^2
\left[
\begin{array}{c}
\frac14 \\
\frac12
\end{array}
\right](0,2\tau)
-
\theta^2
\left[
\begin{array}{c}
\frac34 \\
\frac32
\end{array}
\right](0,2\tau)
\right\}  \\
=&
4
\theta
\left[
\begin{array}{c}
\frac14 \\
1
\end{array}
\right](0,4\tau)
\theta
\left[
\begin{array}{c}
0 \\
0
\end{array}
\right](0,4\tau)
\left\{
\theta^2
\left[
\begin{array}{c}
0 \\
0
\end{array}
\right](0,4\tau)
+
3
\theta^2
\left[
\begin{array}{c}
1 \\
0
\end{array}
\right](0,4\tau)
\right\}. 
\end{align*} 
\par
In the same way, 
we have 
\begin{align*}
&
\theta
\left[
\begin{array}{c}
\frac14 \\
0
\end{array}
\right](0,\tau)
\theta
\left[
\begin{array}{c}
\frac14 \\
\frac12
\end{array}
\right](0,\tau)
\theta
\left[
\begin{array}{c}
\frac14 \\
1
\end{array}
\right](0,\tau)
\theta
\left[
\begin{array}{c}
\frac14 \\
\frac32
\end{array}
\right](0,\tau)  \\
=&
\zeta_8
\theta
\left[
\begin{array}{c}
\frac14 \\
1
\end{array}
\right](0,4\tau)
\theta
\left[
\begin{array}{c}
0 \\
1
\end{array}
\right](0,4\tau)
\left\{
\theta^2
\left[
\begin{array}{c}
0 \\
0
\end{array}
\right](0,4\tau)
-
\theta
\left[
\begin{array}{c}
1 \\
0
\end{array}
\right](0,4\tau)
\right\},  \\
&
\theta
\left[
\begin{array}{c}
0 \\
0
\end{array}
\right](0,\tau)
\theta
\left[
\begin{array}{c}
0 \\
1
\end{array}
\right](0,\tau)
\theta^2
\left[
\begin{array}{c}
1 \\
0
\end{array}
\right](0,\tau)
\theta^2
\left[
\begin{array}{c}
1 \\
\frac12
\end{array}
\right](0,\tau) \\
=&
4
\theta
\left[
\begin{array}{c}
0 \\
0
\end{array}
\right](0,4\tau)
\theta
\left[
\begin{array}{c}
1 \\
0
\end{array}
\right](0,4\tau)
\theta^2
\left[
\begin{array}{c}
0\\
1
\end{array}
\right](0,4\tau) 
\left\{
\theta^2
\left[
\begin{array}{c}
0 \\
0
\end{array}
\right](0,4\tau)
-
\theta^2
\left[
\begin{array}{c}
1 \\
0
\end{array}
\right](0,4\tau)
\right\},  \\
\end{align*}
which proves the proposition. 
\end{proof}

\begin{theorem}
\label{thm:general-Jacobi-1/4-(1)}
{\it
For every $\tau\in\mathbb{H}^2,$ we have 
\begin{align*}
&
\theta^{\prime}
\left[
\begin{array}{c}
1 \\
\frac14
\end{array}
\right](0,\tau)  
=
\frac
{
-
2\pi 
\theta
\left[
\begin{array}{c}
1 \\
\frac14
\end{array}
\right](0,\tau)
\theta^2
\left[
\begin{array}{c}
0 \\
0
\end{array}
\right](0,4\tau)
}
{
\theta^4
\left[
\begin{array}{c}
1 \\
\frac14
\end{array}
\right](0,\tau)
-
\theta^4
\left[
\begin{array}{c}
1 \\
\frac34
\end{array}
\right](0,\tau)
}  \times \\
&\hspace{15mm}
\times
\left(
\theta
\left[
\begin{array}{c}
1 \\
0
\end{array}
\right](0,4\tau)
\theta
\left[
\begin{array}{c}
0 \\
1
\end{array}
\right](0,4\tau)  
\left\{
\theta^2
\left[
\begin{array}{c}
0 \\
0
\end{array}
\right](0,4\tau)
+
3
\theta^2
\left[
\begin{array}{c}
1 \\
0
\end{array}
\right](0,4\tau)
\right\}
+
\theta^4
\left[
\begin{array}{c}
1 \\
\frac34
\end{array}
\right](0,\tau)
\right),
\end{align*}
and
\begin{align*}
&
\theta^{\prime}
\left[
\begin{array}{c}
1 \\
\frac34
\end{array}
\right](0,\tau)  
=
\frac
{
-
2\pi 
\theta
\left[
\begin{array}{c}
1 \\
\frac34
\end{array}
\right](0,\tau)
\theta^2
\left[
\begin{array}{c}
0 \\
0
\end{array}
\right](0,4\tau)
}
{
\theta^4
\left[
\begin{array}{c}
1 \\
\frac14
\end{array}
\right](0,\tau)
-
\theta^4
\left[
\begin{array}{c}
1 \\
\frac34
\end{array}
\right](0,\tau)
}  \times \\
&\hspace{15mm}
\times
\left(
\theta
\left[
\begin{array}{c}
1 \\
0
\end{array}
\right](0,4\tau)
\theta
\left[
\begin{array}{c}
0 \\
1
\end{array}
\right](0,4\tau)  
\left\{
\theta^2
\left[
\begin{array}{c}
0 \\
0
\end{array}
\right](0,4\tau)
+
3
\theta^2
\left[
\begin{array}{c}
1 \\
0
\end{array}
\right](0,4\tau)
\right\}
+
\theta^4
\left[
\begin{array}{c}
1 \\
\frac14
\end{array}
\right](0,\tau)
\right).
\end{align*}
}
\end{theorem}

\begin{proof}
We first note that 
$$
\theta^4
\left[
\begin{array}{c}
1 \\
\frac14
\end{array}
\right](0,\tau)
-
\theta^4
\left[
\begin{array}{c}
1 \\
\frac34
\end{array}
\right](0,\tau)
=
\theta^4
\left[
\begin{array}{c}
1 \\
\frac14
\end{array}
\right](0,\tau)
\left(
1
-
\left\{
\frac
{
\theta
\left[
\begin{array}{c}
1 \\
\frac34
\end{array}
\right](0,\tau)
}
{
\theta
\left[
\begin{array}{c}
1 \\
\frac14
\end{array}
\right](0,\tau)
}
\right\}^4
\right). 
$$
Taking the limit $\tau\longrightarrow i\infty,$ 
we have 
\begin{align*}
1
-
\left\{
\frac
{
\theta
\left[
\begin{array}{c}
1 \\
\frac34
\end{array}
\right](0,i\infty)
}
{
\theta
\left[
\begin{array}{c}
1 \\
\frac14
\end{array}
\right](0,i\infty)
}
\right\}^4
=&
1-
\left(
\frac
{\cos^2 \frac{3\pi}{8}}
{\cos^2 \frac{\pi}{8}}
\right)^2  \\
=&12\sqrt{2}-16\neq 0,
\end{align*}
which implies that 
$$
\theta^4
\left[
\begin{array}{c}
1 \\
\frac14
\end{array}
\right](0,\tau)
-
\theta^4
\left[
\begin{array}{c}
1 \\
\frac34
\end{array}
\right](0,\tau)
\not\equiv 0.
$$
\par
Considering equation (\ref{eqn-log-diff-difference-(1)}) and Proposition \ref{prop:Jacobi-1/4-(1)}, 
we obtain the theorem.
\end{proof}

\begin{theorem}
\label{thm:general-Jacobi-1/4-(2)}
{\it
For every $\tau\in\mathbb{H}^2,$ we have 
\begin{align*}
&
\theta^{\prime}
\left[
\begin{array}{c}
1 \\
\frac14
\end{array}
\right](0,\tau)  
=
\frac
{
-
\pi 
\theta
\left[
\begin{array}{c}
1 \\
\frac14
\end{array}
\right](0,\tau)
\theta
\left[
\begin{array}{c}
0 \\
0
\end{array}
\right](0,4\tau)
}
{
\theta^4
\left[
\begin{array}{c}
1 \\
\frac14
\end{array}
\right](0,\tau)
+
\theta^4
\left[
\begin{array}{c}
1 \\
\frac34
\end{array}
\right](0,\tau)
}  \times \\
&\hspace{0mm}
\times
\Bigg(
\theta^2
\left[
\begin{array}{c}
1 \\
0
\end{array}
\right](0,2\tau)
\theta
\left[
\begin{array}{c}
1 \\
0
\end{array}
\right](0,4\tau)  
\left\{
\theta^2
\left[
\begin{array}{c}
0 \\
0
\end{array}
\right](0,4\tau)
+
3
\theta^2
\left[
\begin{array}{c}
1 \\
0
\end{array}
\right](0,4\tau)
\right\}  \\
&\hspace{50mm}
+
2\sqrt{2}
\theta^4
\left[
\begin{array}{c}
1 \\
\frac34
\end{array}
\right](0,\tau)
\theta
\left[
\begin{array}{c}
0 \\
0
\end{array}
\right](0,2\tau)
\Bigg),
\end{align*}
and
\begin{align*}
&
\theta^{\prime}
\left[
\begin{array}{c}
1 \\
\frac34
\end{array}
\right](0,\tau)  
=
\frac
{
\pi 
\theta
\left[
\begin{array}{c}
1 \\
\frac34
\end{array}
\right](0,\tau)
\theta
\left[
\begin{array}{c}
0 \\
0
\end{array}
\right](0,4\tau)
}
{
\theta^4
\left[
\begin{array}{c}
1 \\
\frac14
\end{array}
\right](0,\tau)
+
\theta^4
\left[
\begin{array}{c}
1 \\
\frac34
\end{array}
\right](0,\tau)
}  \times \\
&\hspace{0mm}
\times
\Bigg(
\theta^2
\left[
\begin{array}{c}
1 \\
0
\end{array}
\right](0,2\tau)
\theta
\left[
\begin{array}{c}
1 \\
0
\end{array}
\right](0,4\tau)  
\left\{
\theta^2
\left[
\begin{array}{c}
0 \\
0
\end{array}
\right](0,4\tau)
+
3
\theta^2
\left[
\begin{array}{c}
1 \\
0
\end{array}
\right](0,4\tau)
\right\}  \\
&\hspace{35mm}
-
2\sqrt{2}
\theta^4
\left[
\begin{array}{c}
1 \\
\frac14
\end{array}
\right](0,\tau)
\theta
\left[
\begin{array}{c}
0 \\
0
\end{array}
\right](0,2\tau)
\Bigg).
\end{align*}
}
\end{theorem}

\begin{proof}
We first note that 
$$
\theta^4
\left[
\begin{array}{c}
1 \\
\frac14
\end{array}
\right](0,\tau)
+
\theta^4
\left[
\begin{array}{c}
1 \\
\frac34
\end{array}
\right](0,\tau)
=
\theta^4
\left[
\begin{array}{c}
1 \\
\frac14
\end{array}
\right](0,\tau)
\left(
1
+
\left\{
\frac
{
\theta
\left[
\begin{array}{c}
1 \\
\frac34
\end{array}
\right](0,\tau)
}
{
\theta
\left[
\begin{array}{c}
1 \\
\frac14
\end{array}
\right](0,\tau)
}
\right\}^4
\right). 
$$
Taking the limit $\tau\longrightarrow i\infty,$ 
we have 
\begin{align*}
1
+
\left\{
\frac
{
\theta
\left[
\begin{array}{c}
1 \\
\frac34
\end{array}
\right](0,i\infty)
}
{
\theta
\left[
\begin{array}{c}
1 \\
\frac14
\end{array}
\right](0,i\infty)
}
\right\}^4
=&
1+
\left(
\frac
{\cos^2 \frac{3\pi}{8}}
{\cos^2 \frac{\pi}{8}}
\right)^2  \\
=&12\sqrt{2}+18\neq 0,
\end{align*}
which implies that 
$$
\theta^4
\left[
\begin{array}{c}
1 \\
\frac14
\end{array}
\right](0,\tau)
+
\theta^4
\left[
\begin{array}{c}
1 \\
\frac34
\end{array}
\right](0,\tau)
\not\equiv 0.
$$
\par
By Lemma \ref{lem:Farkas-Kra}, 
we next note that 
$$
\theta^2
\left[
\begin{array}{c}
1 \\
0
\end{array}
\right](0,\tau)
=2
\theta
\left[
\begin{array}{c}
0 \\
0
\end{array}
\right](0,2\tau)
\theta
\left[
\begin{array}{c}
1 \\
0
\end{array}
\right](0,2\tau).
$$
\par
Considering equation (\ref{eqn-log-diff-difference-(2)}) and Proposition \ref{prop:Jacobi-1/4-(1)}, 
we obtain the theorem.
\end{proof}

\section{More theta constant identities}

\begin{theorem}
{\it
For every $\tau\in\mathbb{H}^2,$ 
we have 
\begin{align}
&\theta 
\left[
\begin{array}{c}
1 \\
0
\end{array}
\right]
\theta^3 
\left[
\begin{array}{c}
1 \\
\frac12
\end{array}
\right]
-
\theta 
\left[
\begin{array}{c}
1 \\
\frac14
\end{array}
\right]
\theta^3 
\left[
\begin{array}{c}
1 \\
\frac34
\end{array}
\right]
-
\theta 
\left[
\begin{array}{c}
1 \\
\frac34
\end{array}
\right]
\theta^3 
\left[
\begin{array}{c}
1 \\
\frac14
\end{array}
\right]=0,    \label{theta-constant-identity(1)}  \\
&\theta^2 
\left[
\begin{array}{c}
1 \\
0
\end{array}
\right]
\theta 
\left[
\begin{array}{c}
1 \\
\frac14
\end{array}
\right]
\theta 
\left[
\begin{array}{c}
1 \\
\frac34
\end{array}
\right]
-
\theta^2 
\left[
\begin{array}{c}
1 \\
\frac14
\end{array}
\right]
\theta^2
\left[
\begin{array}{c}
1 \\
\frac12
\end{array}
\right]
+
\theta^2 
\left[
\begin{array}{c}
1 \\
\frac12
\end{array}
\right]
\theta^2 
\left[
\begin{array}{c}
1 \\
\frac34
\end{array}
\right]=0,    \label{theta-constant-identity(2)}  \\
&\theta^4
\left[
\begin{array}{c}
1 \\
\frac14
\end{array}
\right]
-
\theta^4 
\left[
\begin{array}{c}
1 \\
\frac34
\end{array}
\right]
-
\theta
\left[
\begin{array}{c}
1 \\
\frac12
\end{array}
\right]
\theta^3 
\left[
\begin{array}{c}
1 \\
0
\end{array}
\right]=0.     \label{theta-constant-identity(3)}
\end{align}
}
\end{theorem}

\begin{proof}
We first prove equation (\ref{theta-constant-identity(1)}). 
Equations (\ref{theta-constant-identity(2)}) and (\ref{theta-constant-identity(3)}) can be proved in the same way. 
\par
For this purpose, 
we note that 
$
\dim 
\mathcal{F}_{2}\left[
\begin{array}{c}
\frac24 \\
0
\end{array}
\right] =2$ and 
\begin{alignat*}{2}
&\theta 
\left[
\begin{array}{c}
\frac14 \\
\frac14
\end{array}
\right] (\zeta, \tau)
\theta 
\left[
\begin{array}{c}
\frac14 \\
\frac{7}{4}
\end{array}
\right] (\zeta, \tau), &\quad
&\theta 
\left[
\begin{array}{c}
\frac14 \\
\frac{2}{4}
\end{array}
\right] (\zeta, \tau)
\theta 
\left[
\begin{array}{c}
\frac14 \\
\frac{6}{4}
\end{array}
\right] (\zeta, \tau),  \\ 
&\theta 
\left[
\begin{array}{c}
\frac14 \\
\frac{3}{4}
\end{array}
\right] (\zeta, \tau)
\theta 
\left[
\begin{array}{c}
\frac14 \\
\frac{5}{4}
\end{array}
\right] (\zeta, \tau), &
&\theta^2 
\left[
\begin{array}{c}
\frac14 \\
1
\end{array}
\right] (\zeta, \tau)
\in 
\mathcal{F}_{2}\left[
\begin{array}{c}
\frac24 \\
0
\end{array}
\right].  
\end{alignat*}
Therefore, 
there exists complex numbers $x_1, x_2, x_3, x_4,$ not all zero such that  
\begin{align*}
&x_1\theta 
\left[
\begin{array}{c}
\frac14 \\
\frac{1}{4}
\end{array}
\right] (\zeta, \tau)
\theta 
\left[
\begin{array}{c}
\frac14 \\
\frac{7}{4}
\end{array}
\right] (\zeta, \tau) 
+x_2\theta 
\left[
\begin{array}{c}
\frac14 \\
\frac{2}{4}
\end{array}
\right] (\zeta, \tau)
\theta 
\left[
\begin{array}{c}
\frac14 \\
\frac{6}{4}
\end{array}
\right] (\zeta, \tau) \\ 
&\hspace{30mm}+x_3\theta 
\left[
\begin{array}{c}
\frac14 \\
\frac{3}{4}
\end{array}
\right] (\zeta, \tau)
\theta 
\left[
\begin{array}{c}
\frac14 \\
\frac{5}{4}
\end{array}
\right] (\zeta, \tau) 
+x_4\theta^2 
\left[
\begin{array}{c}
\frac14 \\
1
\end{array}
\right] (\zeta, \tau)=0.
\end{align*}
\par
Substituting 
$$
\zeta=\frac{3\tau\pm3}{8},  \,\, \frac{3\tau\pm2}{8}, \,\, \frac{3\tau\pm 1}{8},  \,\,\frac{3\tau}{8}, 
$$
we obtain 
\begin{equation}
\label{eqn:system-A}
A \mathbf{x}=\mathbf{0},
\end{equation}
where 
\begin{equation*}
A
=
\begin{pmatrix}
0 & 
\theta 
\left[
\begin{array}{c}
1 \\
\frac{1}{4}
\end{array}
\right] 
\theta 
\left[
\begin{array}{c}
1 \\
\frac{3}{4}
\end{array}
\right] & 
\theta 
\left[
\begin{array}{c}
1 \\
0
\end{array}
\right] 
\theta 
\left[
\begin{array}{c}
1 \\
\frac{1}{2}
\end{array}
\right] &
\theta^2 
\left[
\begin{array}{c}
1 \\
\frac{1}{4}
\end{array}
\right]   \\
-\theta 
\left[
\begin{array}{c}
1 \\
\frac{1}{4}
\end{array}
\right] 
\theta 
\left[
\begin{array}{c}
1 \\
\frac{3}{4}
\end{array}
\right] &
0 & 
\theta 
\left[
\begin{array}{c}
1 \\
\frac{1}{4}
\end{array}
\right] 
\theta 
\left[
\begin{array}{c}
1 \\
\frac{3}{4}
\end{array}
\right] & 
\theta^2 
\left[
\begin{array}{c}
1 \\
\frac{1}{2}
\end{array}
\right]   \\
-\theta 
\left[
\begin{array}{c}
1 \\
0
\end{array}
\right] 
\theta 
\left[
\begin{array}{c}
1 \\
\frac{1}{2}
\end{array}
\right] & 
-\theta 
\left[
\begin{array}{c}
1 \\
\frac{1}{4}
\end{array}
\right] 
\theta 
\left[
\begin{array}{c}
1 \\
\frac{3}{4}
\end{array}
\right] & 
0 & 
\theta^2
\left[
\begin{array}{c}
1 \\
\frac{3}{4}
\end{array}
\right] \\
-
\theta^2
\left[
\begin{array}{c}
1 \\
\frac{1}{4}
\end{array}
\right] & 
-
\theta^2
\left[
\begin{array}{c}
1 \\
\frac{1}{2}
\end{array}
\right] & 
-
\theta^2
\left[
\begin{array}{c}
1 \\
\frac{3}{4}
\end{array}
\right] & 
0 
\end{pmatrix},
\end{equation*}
and 
$\mathbf{x}= {}^{t} (x_1, x_2, x_3, x_4).$ 
\par
Since a system of equations (\ref{eqn:system-A}) has a nontrivial solution $\mathbf{x}\neq \mathbf{0},$ 
it follows that $\det A=0,$ 
which proves equation (\ref{theta-constant-identity(1)}). 
\par
Equations (\ref{theta-constant-identity(2)}) and (\ref{theta-constant-identity(3)}) can be proved by considering that 
\begin{alignat*}{2}
&\theta^2 
\left[
\begin{array}{c}
\frac14 \\
0
\end{array}
\right] (\zeta, \tau), 
&\quad
&\theta 
\left[
\begin{array}{c}
\frac14 \\
\frac{2}{4}
\end{array}
\right] (\zeta, \tau)
\theta 
\left[
\begin{array}{c}
\frac14 \\
\frac{6}{4}
\end{array}
\right] (\zeta, \tau),  \\ 
&\theta 
\left[
\begin{array}{c}
\frac14 \\
\frac{3}{4}
\end{array}
\right] (\zeta, \tau)
\theta 
\left[
\begin{array}{c}
\frac14 \\
\frac{5}{4}
\end{array}
\right] (\zeta, \tau), &
&\theta^2 
\left[
\begin{array}{c}
\frac14 \\
1
\end{array}
\right] (\zeta, \tau)
\in 
\mathcal{F}_{2}\left[
\begin{array}{c}
\frac24 \\
0
\end{array}
\right], 
\end{alignat*}
and that
\begin{alignat*}{2}
&\theta^2 
\left[
\begin{array}{c}
\frac14 \\
0
\end{array}
\right] (\zeta, \tau), 
&\quad
&\theta 
\left[
\begin{array}{c}
\frac14 \\
\frac{1}{4}
\end{array}
\right] (\zeta, \tau)
\theta 
\left[
\begin{array}{c}
\frac14 \\
\frac{7}{4}
\end{array}
\right] (\zeta, \tau),  \\ 
&\theta 
\left[
\begin{array}{c}
\frac14 \\
\frac{3}{4}
\end{array}
\right] (\zeta, \tau)
\theta 
\left[
\begin{array}{c}
\frac14 \\
\frac{5}{4}
\end{array}
\right] (\zeta, \tau), &
&\theta^2 
\left[
\begin{array}{c}
\frac14 \\
1
\end{array}
\right] (\zeta, \tau)
\in 
\mathcal{F}_{2}\left[
\begin{array}{c}
\frac24 \\
0
\end{array}
\right].  
\end{alignat*}
\end{proof}

\subsubsection*{Remark}
Since matrix $A$ is a skew-symmetric matrix, 
it follows that equation (\ref{theta-constant-identity(1)}) is expressed in terms of a Pfaffian. 
In the same way, 
equations (\ref{theta-constant-identity(2)}) and (\ref{theta-constant-identity(3)}) can be given in terms of Pfaffian. 
We note that 
in \cite{Matsuda-2} we derived some theta constant identities by considering the determinant structure of a matrix of the theta constants.



\begin{thebibliography}{}
%
%
\bibitem{Berndt}
B. C. Berndt, 
{\it
Number theory in the spirit of Ramanujan, 
}
Stud. Math. Libr., {\bf 34} American Mathematical Society, Providence, RI, 2006.



\bibitem{Dickson}
L. E. Dickson, 
{\it
Modern Elementary Theory of Numbers, 
} 
University of Chicago Press, Chicago, 1939. 



\bibitem{Igusa}
Jun-ichi Igusa, 
{\it On Jacobi's derivative formula and its generalizations,} 
Amer. J. Math.  {\bf 102}  (1980), 409-446. 


\bibitem{Kohler}
G. K\"ohler, 
{\it
Some eta-identities arising from theta series, 
}
Math. Scand. {\bf 66} (1990), 147-154. 


\bibitem{Kohler2}
G. K\"ohler, 
{\it
Eta products and theta series identities, 
}
Springer Monogr. Math. Springer, Heidelberg, 2011.


\bibitem{Farkas-Kra}
H. M. Farkas, and Irwin Kra, 
{\it Theta constants, Riemann surfaces and the modular group,} 
AMS Grad Studies in Math., {\bf 37} 2001.


\bibitem{Farkas-2}
H. M. Farkas, 
{\it Theta functions in complex analysis and number theory,} 
Dev. Math., {\bf 17} (2008), 57-87.


\bibitem{Matsuda-1}
K. Matsuda, 
{\it Generalizations of the Farkas identity for modulus 4 and 7,} 
Proc. Japan Acad. Ser. A Math. Sci.
{\bf 89}  (2013), 129-149.



\bibitem{Matsuda-2}
K. Matsuda, 
{\it The determinant expressions of some theta constant identities,} 
Ramanujan J.  {\bf 34}  (2014),  449-456. 









\bibitem{Zucker1}
I. J. Zucker, 
{\it
A systematic way of converting infinite series into infinite products, 
}
J. Phys. A  {\bf 20}  (1987),  L13-L17. 

\bibitem{Zucker2}
I. J. Zucker, 
{\it
Further relations amongst infinite series and products. II. The evaluation of three-dimensional lattice sums, 
}
J. Phys. A  {\bf 23}  (1990), 117-132. 


\end{thebibliography}


\end{document}